\documentclass{amsart}
\usepackage{amssymb}
\usepackage{verbatim}
\usepackage[usenames]{color}
\usepackage{hyperref}
\usepackage{url}
\usepackage[all]{xy}

\newtheorem{theorem}{Theorem}[section]
 
\newtheorem{lemma}[theorem]{Lemma}

\newtheorem{fact}[theorem]{Fact}

\theoremstyle{definition}
\newtheorem{definition}[theorem]{Definition}

\newtheorem{remark}[theorem]{Remark}

\newtheorem{question}[theorem]{Question}
\newtheorem{conjecture}[theorem]{Conjecture}

\theoremstyle{remark}

\newcommand{\defemph}{\textit}

\newcommand{\dom}{\textrm{Dom}}

\newcommand{\zfc}{\textrm{ZFC}}

\newcommand{\ad}{\textrm{AD}}

\newcommand{\mc}{\mathcal}
\newcommand{\mbb}{\mathbb}

\newcommand{\forces}{\Vdash}

\newcommand{\al}{\alpha}
\newcommand{\om}{\omega}

\newcommand{\sse}{\subseteq}

\DeclareMathOperator{\lh}{lh}

\newcommand{\re}{\restriction}
\newcommand{\bP}{\mathbb{P}}

\newcommand{\bQ}{\mathbb{Q}}

\newcommand{\ra}{\rightarrow}

\newcommand{\lgl}{\langle}
\newcommand{\rgl}{\rangle}

\newcommand{\Erdos}{Erd{\H{o}}s}
\newcommand{\Lauchli}{L{\"{a}}uchli}

\theoremstyle{remark}

\newcommand{\noprint}[1]{\relax}

\title{The Halpern-L\"auchli Theorem at a Measurable Cardinal}

\author{Natasha
  Dobrinen}
\address{Department of Mathematics\\
  University of Denver \\
   2280 S Vine St\\ Denver, CO \ 80208 U.S.A.}  
\email{Natasha.Dobrinen@du.edu}
\urladdr{\url{http://web.cs.du.edu/~ndobrine}} 
\thanks{The first author was partially supported by  National Science Foundation Grant DMS-1301665 and the Isaac Newton Institute.
This research was partially done whilst the first author was a visiting fellow at the Isaac Newton Institute for Mathematical Sciences in the programme `Mathematical, Foundational and Computational Aspects of the Higher Infinite' (HIF) funded by EPSRC grant EP/K032208/1.}

\author{Dan Hathaway}
\address{Mathematics Department\\
University of Denver\\
   2280 S Vine St\\ Denver, CO \ 80208 U.S.A.}  
\email{Daniel.Hathaway@du.edu}
\urladdr{\url{http://mysite.du.edu/~dhathaw2/}}

\begin{document}

\begin{abstract}
Several variants  of the Halpern-\Lauchli\ Theorem  for   trees of uncountable height
are investigated.
For $\kappa$ weakly compact,  we prove that the various statements are all equivalent.
We show that the strong tree version holds for one tree on any infinite cardinal.
For any finite $d\ge 2$,
we prove the consistency of the Halpern-\Lauchli\ Theorem on $d$ many normal
 $\kappa$-trees at a measurable cardinal $\kappa$, 
given the consistency of a  $\kappa+d$-strong cardinal.
This follows from a   more general consistency result at measurable $\kappa$, which
includes the possibility of   infinitely many trees, assuming partition relations which hold in models of AD.
\end{abstract}

\maketitle

\section{Introduction}

Halpern and \Lauchli\ proved their celebrated theorem regarding Ramsey theory on products of trees in \cite{Halpern/Lauchli66} as a necessary step for their  construction in \cite{Halpern/Levy71} of a model of ZF in which the Boolean Prime Ideal Theorem holds but the Axiom of Choice fails.
Since then, many variations have been established and applied.
One of the earliest of these is due to Milliken, who   extended the  Halpern-\Lauchli\ Theorem
to colorings of finite products of strong trees in  \cite{Milliken79}.
As the two versions are equivalent, this version is often used synonymously with Halpern and \Lauchli's orginal version. 
Milliken further proved 
in  \cite{Milliken81} that 
the collection of strong trees 
forms, in modern terminology, a topological Ramsey space.
Further   variations and applications include 
$\om$ many perfect trees in \cite{Laver84};
partitions of products in \cite{DiPrisco/Llopis/Todorcevic04};
the density version in
\cite{Dodos/Kanellopoulos/Tyros13};
the dual version  in \cite{Todorcevic/Tyros13};
 canonical equivalence relations on finite strong trees in \cite{Vlitas14};
and
applications to  colorings of  subsets of the rationals in \cite{DevlinThesis80} and 
 \cite{Vuksanovic03}, and to
finite
Ramsey degrees  of the Rado graph  in \cite{Sauer06} which in turn was applied to show the Rado graph has the rainbow Ramsey property in \cite{Dobrinen/Laflamme/Sauer16}, to name just a few.

Shelah produced the first result  generalizing the Halpern-\Lauchli\ Theorem to an uncountable height tree in 
\cite{Shelah91}.
There, in Lemma 4.1, Shelah 
 proved that for certain models of ZFC with a measurable cardinal $\kappa$, given any finite $m$ and any coloring 
of the $m$-sized subsets of the levels of the tree ${}^{<\kappa}2$, (that is,  a coloring of $\bigcup_{\zeta<\kappa}[{}^{\zeta}2]^m$),
 into less than $\kappa$ many colors,
there is a strong subtree (see Definition \ref{def.strongsubtree})
$T\sse {}^{<\kappa}2$
on which the coloring takes  only finitely many colors.
Shelah's
 proof 
builds on and extends Harrington's forcing argument for the Halpern-\Lauchli\ Theorem on ${}^{<\om}2$,
and  assumes that $\kappa$ is a cardinal which  is  measurable after adding $\lambda$ many Cohen subsets of $\kappa$, where $\lambda$ is large enough that the partition relation $\lambda\ra(\kappa^+)^{2m}_{2^{\kappa}}$ holds.
Thus, his result also holds for any $\kappa$ which  is supercompact after a Laver treatment.

In \cite{Dzamonja/Larson/MitchellQ09}, 
Dzamonja, Larson and Mitchell extended Shelah's proof to include colorings of all antichains in the tree  ${}^{<\kappa}2$ of a fixed finite size $m$, rather than just colorings of subsets of the same levels.
They proved that given a coloring of $m$-sized antichains of the tree ${}^{<\kappa}2$ with less than $\kappa$ many colors,
there is a strong subtree $T$ isomorphic to ${}^{<\kappa}2$ on which the set of $m$-sized antichains takes on only finitely many colors.
These finitely many colors are classified by the embedding types of the trees induced by the antichains.
They then applied this result to  find the Ramsey degrees of $<\kappa$ colorings of $[\mathbb{Q}_{\kappa}]^m$, $m$ any finite integer, and of $<\kappa$ colorings of the copies of a fixed finite graph inside  the  Rado graph on $\kappa$ many vertices in  \cite{Dzamonja/Larson/MitchellQ09}  and \cite{Dzamonja/Larson/MitchellRado09}, respectively.
Here $\mathbb{Q}_\kappa$ is the unique, up to isomorphism, $\kappa$-dense
 linear order of size $\kappa$.
A poset $\langle L, <_L \rangle$ is $\kappa$-dense iff whenever
 $A, B \subseteq L$ are such that $a < b$ for all $a \in A$ and $b \in B$,
 then there is some $c \in L$ such that
 $a < c < b$ for all $a \in A$ and $b \in B$.

This work has left open the following questions. 

\begin{question}\label{q1}
For which uncountable cardinals $\kappa$ can
 the Halpern-\Lauchli\ Theorem hold for  trees of height $\kappa$, either in ZFC or consistently?
\end{question}

\begin{question}\label{q2}
What is the consistency strength of the Halpern-\Lauchli\ Theorem for  $\kappa$ uncountable; in particular for $\kappa$ a measurable cardinal? 
\end{question}

\begin{question}\label{q3}
Given a fixed number of trees, 
are the  weaker (somewhere dense)  and stronger (strong tree) forms of the Halpern-\Lauchli\ Theorem  equivalent, for $\kappa$ uncountable?
\end{question}

In Section \ref{sec.3}, we answer Question \ref{q3} for the case of weakly compact cardinals:  Theorem \ref{thm.SDHLimpliesHL} shows that 
when $\kappa$ is weakly compact, all the various weaker and stronger forms of the Halpern-\Lauchli\ Theorem on $\delta$ many regular $\kappa$ trees (see Section \ref{sec.2}) are equivalent,
where $\delta$ can  be any cardinal less than $\kappa$.
In Fact \ref{fact.1treeZFC}, we show that 
the somewhere dense version for finitary colorings on one tree on any cardinal is a ZFC result.
In fact,
if $\kappa$ is an infinite cardinal, then the strong tree version of the Halpern-\Lauchli\ Theorem holds for finitary colorings of  one regular  tree on $\kappa$, thus 
 answering Question \ref{q1} for the case of  one tree and finitely many colors.

In Section \ref{sec.4}, for any fixed positive integer $d$,
Theorem \ref{thm.HLmbl} shows 
that it is consistent  for the Halpern-\Lauchli\ Theorem to hold on $d$ many regular trees on a  measurable cardinal $\kappa$.
Further, this theorem
yields that a  $\kappa+d$-strong
cardinal is
an upper bound 
 for the consistency strength of the Halpern-\Lauchli\ Theorem for $d$ many  regular trees on a measurable cardinal $\kappa$, 
thus partially answering Questions \ref{q1} and \ref{q2}.
This is weaker than the upper bound of a cardinal $\kappa$ which is  $(\kappa+2d)$-strong, that would be needed if one simply lifted the standard version of  Harrington's forcing proof to a measurable cardinal.
This result also presents an interesting contrast to the $(\kappa +2d+2)$-strong cardinal mentioned as an upper bound for the consistency strength of 
 Lemma 4.1 in \cite{Shelah91}
and its strengthening
Theorem 2.5 in 
 \cite{Dzamonja/Larson/MitchellQ09}, which colors antichains of size $d$ in one tree on $\kappa$.
We conjecture in Section \ref{sec.6} that the consistency strength of the Halpern-Lauchli\ Theorem  on $d$ many regular trees  for $\kappa$ measurable is in fact a $\kappa +d$-strong cardinal.
Theorem \ref{thm.HLmbl} follows from a more general result,  
Theorem \ref{thm.generalHL},
which includes the possibility of infinitely many trees assuming certain partition relations
which hold assuming $\ad$.
Section \ref{sec.6} concludes with open problems and conjectures and their relationships with previous work.

The  authors wish to thank 
James Cummings, Mirna Dzamonja, Jean Larson, Paul Larson, and  Bill Mitchell
 for very helpful conversations.
The first author gratefully acknowledges and is indebted to Richard Laver for outlining Harrington's forcing  proof to her in 2011.


\section{Variants of the Halpern-\Lauchli\  Theorem  and simple  implications}\label{sec.2}

This section contains the relevant definitions, various versions of the Halpern-\Lauchli\ Theorem, and the immediate implications between them.
A tree   $T \subseteq {^{<\kappa}\kappa}$
is a {\em $\kappa$-tree} if $T$ has cardinality $\kappa$ and every level of $T$ has cardinality less than $\kappa$.
$T$ is {\em perfect} if for each node $t$ in $T$ there is an extension of $t$ which splits in $T$.
We shall call a  tree $T \subseteq {^{<\kappa}\kappa}$
  \defemph{regular} if it is a perfect
 $\kappa$-tree in which every maximal branch has
 cofinality $\kappa$.
Given any subset
 $T \subseteq {^{<\kappa} \kappa}$
 and $\zeta < \kappa$,
 we write
 $$T(\zeta) := T \cap {^\zeta \kappa}$$
 for the set of nodes on the $\zeta$-th level of $T$.
Given $t \in T$,
 $$T[t] := \{ s \in T : s \sqsubseteq t \mbox{ or } 
 t \sqsubseteq s \}.$$
In this paper, the successor of a node
 always means immediate successor.

Next, we define the notion of  strong subtree, originally defined by Milliken in \cite{Milliken79}.

\begin{definition}\label{def.strongsubtree}
Let $T \subseteq {^{<\kappa} \kappa}$ be regular.
A tree $T' \subseteq T$
 is a \defemph{strong subtree} of $T$
 as witnessed by some set
 $A \subseteq \kappa$ cofinal in $\kappa$ if and only if
$T'$ is regular and for each $t\in T'(\zeta)$ for $\zeta<\kappa$,
$\zeta\not\in A$ implies that $t$ has a unique successor in $T'$ on level $\zeta+1$,
and $\zeta\in A$ implies every successor of $t$ in $T$ is also in $T'$.
\end{definition}

Given an ordinal $\delta > 0$ and a sequence
 $\langle X_i \subseteq {^{<\kappa}\kappa} :
 i < \delta \rangle$, define
 $$\bigotimes_{i < \delta} X_i :=
 \{ \langle x_i : i < \delta \rangle :
 (\exists \zeta < \kappa)(\forall i < \delta)\,
 x_i \in X_i(\zeta) \},$$
 the \textit{level-product} of the $X_i$'s.
We will call a set of nodes all on the same level a
 \textit{level set}.
Similarly, we will call a sequence of nodes or a sequence of
 sets of nodes a \textit{level sequence}
 if all nodes are on the same level.

The following is  the strong tree  version of the
 Halpern-L\"auchli Theorem, which we shall denote by 
 $\textrm{HL}(\delta,\sigma,\kappa)$.

\begin{definition}
For $\delta, \sigma > 0$ ordinals
 and $\kappa$ an infinite cardinal,
 $\textrm{HL}(\delta,\sigma,\kappa)$
 is the following statement:
Given any sequence
 $\langle T_i \subseteq {^{<\kappa}\kappa} :
 i < \delta \rangle$
 of regular trees and a coloring
 $c:\bigotimes_{i<\delta} T_i\ra\sigma$,
 there exists a sequence of trees
 $\langle T'_i : i < \delta \rangle$
 such that
\begin{enumerate}
\item
 each $T'_i$ is a strong subtree of $T_i$
 as witnessed by the same set $A \subseteq \kappa$
 independent of $i$,
 and
\item
 there is some $\sigma' < \sigma$
  such that for each $\zeta \in A$,
 $c`` \bigotimes_{i<\delta} T'_i(\zeta) = \{\sigma'\}$.
\end{enumerate}
\end{definition}

Given $t \in {^{<\kappa}\kappa}$,
 we define
 $$\textrm{Cone}(t) :=
 \{ s \in {^{<\kappa} \kappa} : s \sqsupseteq t \}.$$
We say that $X \subseteq {^{<\kappa}\kappa}$
 \textit{dominates} $Y \subseteq {^{{<\kappa}}\kappa}$
 if and only if $$(\forall y \in Y)(\exists x \in X)\,
 y \sqsubseteq x.$$
We say $X \subseteq {^{<\kappa}\kappa}$ {\em dominates}
 $y \in {^{<\kappa}\kappa}$ just when
 $X$ dominates $\{y\}$.
We now give the definition of the
 Somewhere-Dense Halpern-L\"auchli Theorem.

\begin{definition}
\label{sdhl_def}
For nonzero ordinals $\delta$ and $\sigma$, and an infinite cardinal $\kappa$,
 $\textrm{SDHL}(\delta,\sigma,\kappa)$
 is the statement that given any sequence
 $\langle T_i \subseteq {^{<\kappa}\kappa} : i < \delta
  \rangle$
 of regular trees
 and any coloring $$c :
 \bigotimes_{i < \delta} T_i
 \to \sigma,$$
 there exist
 $\zeta < \zeta' < \kappa$,
 $\langle t_i \in T_i(\zeta) : i < \delta \rangle$, and
 $\langle X_i \subseteq T_i(\zeta') : i < \delta \rangle$
 such that each $X_i$ dominates
 $T_i(\zeta+1) \cap \textrm{Cone}(t_i)$
 and $$|c`` \bigotimes_{i < \delta} X_i| = 1.$$
\end{definition}


We point out the following simple fact.

\begin{fact}\label{fact.1treeZFC}
For every infinite cardinal $\kappa$ and each positive integer $k$,
SDHL$(1,k,\kappa)$ holds.
\end{fact}

\begin{proof}
If $k=1$, the result is, so assume $k\ge 2$.
To prove SDHL$(1,k,\kappa)$,
let $T$ be a regular subtree of ${}^{<\kappa}{\kappa}$ and $c$ be a coloring from the nodes in $T$ into $k$.  
If  there exist a   node $t\in T$   and a
level set $X\sse T$ dominating   the successors of $t$ such that every node in $X$ has color $0$, then we are done.
Otherwise, for each node $t\in T$ and each level set $X\sse T$ dominating the successors of $t$,   $c``X\cap\{1,\dots,k-1\}\ne\emptyset$.
If  there exist a   node $t\in T$   and a
level set $X\sse T$ dominating   the successors of $t$ such that every node in $X$ has color $1$, then we are done.
Otherwise, for each node $t\in T$ and each level set $X\sse T$ dominating the successors of $t$,   $c``X\cap\{2,\dots,k-1\}\ne\emptyset$.
Continuing in this manner, either there is a $j<k-1$ and some node $t$ with a level set $X$ dominating the successors of $t$ each with $c$-color $j$,
or else, 
for each node $t\in T$ and each level set $X\sse T$ dominating the successors of $t$,   $c``X\cap\{k-1\}\ne\emptyset$.
In this case,
choose any node $t\in T$ and list the successors of $t$ in $T$ as $\lgl s_i:i<\eta\rgl$, where  $\eta<\kappa$.
For each $i<\eta$, let $Y_i$ denote the set of all 
the successors of $s_i$.
Then $Y_i$ 
  forms a level set dominating the successors of   $s_i$, 
so $c`` Y_i\cap\{k-1\}\ne\emptyset$.
For each $i<\eta$, take one  $t_i\in Y_i$ such that $c(t_i)=k-1$.
Then the set $\{t_i:i<\eta\}$ is a level set dominating every successor of  $t$ and is monochromatic in color $k-1$.
\end{proof}

\begin{definition}
Given a regular tree $T$ and
 a level $\zeta$ less than the height of $T$,
 a set $X \subseteq T$ is $\zeta$-\defemph{dense}
 if and only if $X$ dominates $T(\zeta)$.
Given a sequence of regular trees
 $\langle T_i : i < \delta \rangle$
 and $\vec{x}
 = \langle x_i : i < \delta \rangle
 \in \bigotimes_{i < \delta} T_i$,
 a sequence of sets $X_i \subseteq T_i$ for
 $i < \delta$ is
 $\zeta$-$\vec{x}$-\defemph{dense} if and only if for each $i < \delta$,
 $X_i$ dominates $T_i(\zeta) \cap
 \textrm{Cone}(x_i)$.
\end{definition}

The definition of
 $\textrm{SDHL}(\delta,\sigma,\kappa)$ can be weakened
 by not requiring the sets
 $X_i$ for $i < \delta$ to be level sets; in this case the colorings must color the full
 product $\prod_{i<\delta}T_i$.
Call this weakening
 $\textrm{SDHL}'(\delta,\sigma,\kappa)$.
A \defemph{somewhere dense matrix} for a sequence
 of regular trees $\langle T_i : i < \delta \rangle$
 is a sequence of sets
 $\langle X_i \subseteq T_i : i < \delta \rangle$
 such that there are nodes
 $t_i \in T_i$ for $i < \delta$ all of the same length
 such that each $X_i$ dominates all successors
 of $t_i$ in $T_i$.
Thus, $\textrm{SDHL}'(\delta,\sigma,\kappa)$ is
 the statement that any coloring
 $c : \prod_{i < \delta} T_i \to \sigma$
 is constant on $\prod_{i < \delta} X_i$
 for some somewhere dense matrix
 $\langle X_i \subseteq T_i : i < \delta \rangle$.
Certainly $\textrm{SDHL}(\delta,\sigma,\kappa)$
 implies $\textrm{SDHL}'(\delta,\sigma,\kappa)$.
Our definition of a somewhere dense matrix
 is not exactly the same as in \cite{TodorcevicBK10},
 but both versions yield equivalent definitions of
 $\mbox{SDHL}'(\delta, \sigma, \kappa)$,
 which is the point of the next remark.

\begin{remark}
The $\textrm{SDHL}'(\delta,\sigma,\kappa)$ is in fact equivalent to the statement obtained by making the following
 modification.
The modification is to allow the nodes $t_i$ to come from different levels of the $T_i$,
 while requiring that the supremum $\eta$
 of the length of the members of $\{t_i : i < \delta\}$
 to be strictly less than the length of any member of
 $\bigcup \{X_i : i < \delta\}$,
 and requiring each $X_i$ to dominate
 $\mbox{Cone}(t_i) \cap T(\eta + 1)$ instead of
 the set of successors of $t_i$.
It is routine to check that the statement under this modification
 is equivalent to $\textrm{SDHL}'(\delta,\sigma,\kappa)$.
\end{remark}

By the definitions and previous discussions, the following implications are immediate.

\begin{fact}\label{fact.simpleimplications}
The following statements are arranged from strongest to weakest, where $\delta$ and $\sigma$ are assumed to be strictly less than $\kappa$.
\begin{enumerate}
\item
HL$(\delta,\sigma,\kappa)$;
\item
SDHL$(\delta,\sigma,\kappa)$; 
\item
SDHL$'(\delta,\sigma,\kappa)$; 
\item
SDHL$'(\delta,\sigma,\kappa)$
 where the $t_i$ need not be from the same level.
\end{enumerate}
\end{fact}

The remark above shows that in fact (4) implies (3).
In the next section will show that
 (3) implies (1).
Thus, all four statements are equivalent.


\section{The various forms of Halpern-\Lauchli\ are equivalent, for $\kappa$ weakly compact}\label{sec.3}

This section provides proofs that all the versions of the Halpern-\Lauchli\ Theorem stated in the previous section are equivalent, provided that $\kappa$ is weakly compact.

First, we prove that if $\kappa$ is weakly compact
 and $\delta, \sigma < \kappa$,
 then $\textrm{SDHL}'(\delta,\sigma,\kappa)$ implies
 $\textrm{SDHL}(\delta,\sigma,\kappa)$.
The proof proceeds via the next three lemmas.
Recall that a cardinal $\kappa$ is {\em weakly compact} if $\kappa$ is strongly inaccessible and satisfies the Tree Property at $\kappa$; that is, every $\kappa$-tree has a cofinal branch.

\begin{lemma}\label{lem.SDHLfromSDHL'1}
Suppose  $\kappa$ is weakly compact.
Let $\delta,\sigma < \kappa$ be  non-zero ordinals, and 
let $\langle T_i \subseteq {^{<\kappa} \kappa} :
 i < \delta \rangle$ be a sequence of regular trees.
Let $U$ be the tree of partial colorings of
 $\prod_{i < \delta} T_i$ into $\sigma$ many colors, 
where a node of $U$
 on level $\alpha$ corresponds to a coloring of
 all tuples $\langle t_i \in T_i : i < \delta \rangle$
 where $\sup \{ \lh(t_i) : i < \delta \} < \alpha$.

Suppose $Z \subseteq U$ has size $\kappa$
 and each $z \in Z$ corresponds to a coloring
 $c_z$ which  is not monochromatic on any 
somewhere dense  matrix  contained in its domain.
Then there is a coloring $c$ of all of
 $\prod_{i < \delta} T_i$ such that 
$c$ is not monochromatic on any 
 somewhere dense matrix.
\end{lemma}

\begin{proof}
Note that if $\langle u_\alpha : \alpha < \beta \rangle$
 is a strictly increasing sequence of nodes of $U$
 for some limit ordinal $\beta$,
 then there is a unique node on level
 $\sup_{\alpha < \beta} \lh(u_\alpha)$
 that extends each $u_\alpha$.
Furthermore, note that $U$ is a $\kappa$ tree.
Fix a set $Z \subseteq U$ as in the hypotheses.

Let $S$ be the set of all predecessors of
 elements of $Z$.
Then  $S$ is a $\kappa$-tree.
By the weak compactness of $\kappa$, 
 fix a length $\kappa$ branch through $S$.
This branch corresponds to a coloring $c$
 of all of $\prod_{i<\delta} T_i$.
Now, for each $\alpha < \kappa$ there is a node
 $z \in Z$ such that the
domain of $c_z$ includes
 $ \prod_{i<\delta}\bigcup_{\zeta<\al}T_i(\zeta)$
 and both $c$ and $c_z$
 color those sequences the same way.
For each $\al<\kappa$, choose one such node and label it  $z_{\al}$.

Now, pick any somewhere dense
 $\langle X_i \subseteq T_i : i < \delta \rangle$
 such that there is some $\alpha < \kappa$ satisfying
 $(\forall i < \delta)
  (\forall x \in X_i)\, \lh(x) < \alpha$.
Fix such an $\alpha < \kappa$.
We must show that
 $|c``\prod_{i<\delta}X_i| > 1$.
We have that for all $\langle x_i \in X_i : i < \delta \rangle$,
 $$c(\langle x_i : i < \delta \rangle) =
 c_{z_\al}(\langle x_i : i < \delta \rangle).$$
By the hypothesis on $c_{z_\al}$, we have
 $|c_{z_\al}`` \prod_{i < \delta} X_i| > 1$.
Thus, $|c`` \prod_{i < \delta} X_i| > 1$ as desired.
\end{proof}

\begin{lemma}\label{lem.SDHLfromSDHL'2}
Suppose $\kappa$ is weakly compact.
Let $\sigma, \delta > 0$ be ordinals and assume
 $\delta, \sigma < \kappa$.
Assume $\textrm{SDHL}'(\delta,\sigma,\kappa)$ holds.
Let $\langle T_i \subseteq {^{<\kappa}\kappa}
 : i < \delta \rangle$ be a sequence of regular trees.
Then there is a level $1\le \zeta < \kappa$ such that for every coloring
 $$c : \prod_{i < \delta}
 \bigcup_{\xi < \zeta} T_i(\xi)
 \to \sigma,$$
 there is a somewhere dense matrix
 $\langle X_i \subseteq
 \bigcup_{\xi < \zeta} T_i(\xi) : i < \delta \rangle$
 such that $c$ is constant on
 $\prod_{i < \delta} X_i$.
\end{lemma}

\begin{proof}
Assume there is no such level $\zeta$.
Then for each $\zeta < \kappa$ we may pick a coloring
 $c_\zeta : \prod_{i < \delta}
 \bigcup_{\xi < \zeta} T_i(\xi) \to \sigma$ 
which is not constant on any
 somewhere dense matrix in the domain of $c_\zeta$.
Letting $U$ be the tree of partial colorings
 described in Lemma \ref{lem.SDHLfromSDHL'1}, the  set
 $Z =\{c_\zeta:\zeta<\kappa\}$ is a subset of $ U$ of size $\kappa$.
Applying Lemma \ref{lem.SDHLfromSDHL'1}, 
 we obtain a coloring
 $c : \prod_{i < \delta} T_i \to \sigma$
 for which there is no somewhere dense matrix
 $\langle X_i \subseteq T_i : i < \delta \rangle$
 such that $|c``\prod_{i<\delta} X_i| = 1$.
Hence, $\textrm{SDHL}'(\delta,\sigma,\kappa)$
 fails, which is a contradiction.
\end{proof}

The following bounded version is the analogue of the Finite Halpern-\Lauchli\ Theorem for finitely many trees on $\om$ (see, for instance, Theorem 3.9 in \cite{TodorcevicBK10}).

\begin{lemma}\label{lem.SDHLfromSDHL'3}
\label{sdhl_holding_lemma}
Suppose $\kappa$ is weakly compact.
Let $\sigma, \delta <\kappa$ be nonzero ordinals and 
assume $\textrm{SDHL}'(\delta,\sigma,\kappa)$ holds.
Let $\langle T_i \subseteq {^{<\kappa}\kappa}
 : i < \delta \rangle$ be a sequence of regular trees.
Then there is a level $1\le \zeta < \kappa$ such that  for every coloring
 $$\bar{c} : \bigotimes_{i < \delta}
 T_i(\zeta) \to \sigma,$$
 there is a somewhere dense matrix
 $\langle Y_i \subseteq T_i(\zeta) : i < \delta \rangle$
 such that $\bar{c}$ is constant on
 $\bigotimes_{i < \delta} Y_i$.
\end{lemma}

\begin{proof}
Fix some  $\zeta < \kappa$ for which the  Lemma  \ref{lem.SDHLfromSDHL'2} holds.
For each $i < \delta$ and each node
 $t \in
 \bigcup_{\xi < \zeta} T_i(\xi)$,
 associate a node $f_i(t) \in T_i(\zeta)$
 such that $f_i(t) \sqsupseteq t$.
Now fix a coloring $\bar{c} : \bigotimes_{i < \delta}
 T_i(\zeta) \to \sigma$.
Let $c : \prod_{i < \delta}
 \bigcup_{\xi < \zeta} T_i(\xi) \to \sigma$
 be defined by $$c( \langle t_i : i < \delta \rangle )
 := \bar{c}(\langle f_i(t_i) : i < \delta \rangle).$$
By the property of $\zeta$,
 there is a somewhere dense matrix
 $\langle X_i \subseteq
 \bigcup_{\xi < \zeta} T_i(\xi)
 : i < \delta \rangle$
 such that $|c``\prod_{i<\delta}X_i| = 1$.
For each $i < \delta$, let
 $$Y_i := \{ f_i(t) : t \in X_i \}.$$
The sequence
 $\langle Y_i \subseteq T_i(\zeta) : i < \delta \rangle$
 is a somewhere dense matrix,
 and by the definition of $c$ we have
 $|\bar{c}``\bigotimes_{i<\delta}Y_i| = 1$.
\end{proof}

\begin{remark}
We point out that the set of $\zeta<\kappa$ satisfying Lemma \ref{lem.SDHLfromSDHL'2}  is closed upwards, and hence Lemma \ref{lem.SDHLfromSDHL'3} also holds for $\zeta$
 in a final segment of $\kappa$.
\end{remark}

\begin{theorem}\label{thm.SDHL=SDHL'}
Suppose $\kappa$ is weakly compact.
Let $\sigma, \delta > 0$ be ordinals and assume
 $\delta, \sigma < \kappa$.
Then
 $\textrm{SDHL}'(\delta,\sigma,\kappa)$ holds if and only if  $\textrm{SDHL}(\delta,\sigma,\kappa)$ holds.
\end{theorem}
\begin{proof}
This follows by the lemmas above.
\end{proof}


Jing Zhang recognized that the proof that $\textrm{SDHL}$
 implies $\textrm{HL}$ in\
 \cite{TodorcevicBK10} works for an arbitrary
 inaccessible $\kappa$, even one that is not weakly compact.

\begin{lemma}
\label{expanding_one_line}
Let $\kappa$ be a cardinal and
 let $\delta, \sigma < \kappa$ be non-zero ordinals.
Let $\langle T_i : i < \delta \rangle$
 be a sequence of regular trees.
Fix $c : \bigotimes_{i < \delta} T_i \to \sigma$.
Suppose there is a level sequence
 $\vec{x} = \langle x_i \in T_i : i < \delta \rangle$ such that
 for each $\xi < \kappa$, there is a level sequence
 $\langle X_i \subseteq T_i(\zeta) : i < \delta \rangle$
 for some $\zeta$ such that
 $\langle X_i : i < \delta \rangle$ is
 $\xi$-$\vec{x}$-dense and
 $|c``\bigotimes_{i<\delta}X_i| = 1$.
Then the conclusion of $\textrm{HL}(\delta,\sigma,\kappa)$
 holds.
\end{lemma}

\begin{proof}
We will construct a sequence of strong subtrees
 $\langle T_i' \subseteq T_i : i < \delta \rangle$,
 as witnessed by the same set $A \subseteq \kappa$
 independent of $i$, such that for each $\zeta \in A$,
 there is some $\sigma_\zeta < \sigma$ such that
 $c``\bigotimes_{i < \delta} T_i'(\zeta)
 = \{ \sigma_\zeta \}$.
Then, by the pigeonhole principal, there will be
 $\kappa$ many $\sigma_\zeta$ that are equal to the
 same ordinal, call it $\sigma'$.
Let $\tilde{A} \subseteq A$ be the set of levels
 associated to the color $\sigma'$.
We may then thin each $T_i'$ to a strong subtree
 $T_i''$ as witnessed by $\tilde{A}$
 (independent of $i$) such that
 for each $\zeta \in \tilde{A}$,
 $c``\bigotimes_{i < \delta} T_i''(\zeta) =
 \{ \sigma' \}$.
This is the conclusion of
 $\textrm{HL}(\delta,\sigma,\kappa)$.

For each $i < \delta$,
 instead of directly constructing $T_i'$,
 we will construct a subset $S_i$ of $T_i'$
 and $T_i'$ will be the set of all initial
 segments of elements of $S_i$.
Each $S_i$ will be the union of level sets:
 $S_i := \bigcup_{\zeta \in A} L_{i,\zeta}$
 where each $L_{i,\zeta}$ will be a subset of
 $T_i(\zeta)$.
At the same time, we will construct $A \subseteq \kappa$.
We will have it so
 $(\forall \zeta \in A)\,
 |c``\bigotimes_{i<\delta} L_{i,\zeta}| = 1$.
Initially the set $A$ is empy and
 no $L_{i,\zeta}$'s have been defined.

Now assume we are at some stage of the construction.
There are three cases:

\underline{Case 1}:
No $L_{i,\zeta}$'s have been constructed so far
 (and $A$ is empty).
In this case, set $\xi$ to be the level of the $x_i$
 and for each $i < \delta$,
 let $U_{i,\xi+1} := T_i(\xi + 1) \cap
  \mbox{Cone}(x_i)$.

\underline{Case 2}:
Some $L_{i,\zeta}$'s have been constructed and
 there is a largest $\xi < \kappa$ for which
 some $L_{i,\xi}$ exists.
Fix this $\xi$.
For each $i < \delta$, define
 $U_{i,\xi+1} := T_i(\xi+1) \cap
 (\bigcup_{t \in L_{i,\xi}} \mbox{Cone}(t) )$.

\underline{Case 3}:
The set of $\eta$'s for which the
 $L_{i,\eta}$'s exist is below but cofinal in
 some fixed $\xi < \kappa$.
Let $W_i \subseteq T_i(\xi)$
 be the set of $t \in T_i(\xi)$ such that
 the set of $\eta < \xi$ such that
 $t \restriction \eta \in L_{i,\eta}$
 is cofinal in $\xi$.
For each $i < \delta$, define
 $U_{i,\xi+1} := T_i(\xi + 1) \cap
 (\bigcup_{t \in W} \mbox{Cone}(t))$.

Assuming one of the three cases above holds,
 we have both an ordinal $\xi$ and a set
 $U_{i,\xi+1} \subseteq T_i(\xi+1)$
 for each $i < \delta$.
Apply the hypothesis of the lemma to get a level
 sequence $\langle X_i : i < \delta \rangle$
 that is $(\xi+1)$-$\vec{x}$-dense and
 $|c``\bigotimes_{i<\delta}X_i| = 1$.
For each $i < \delta$,
 let $X_i' \subseteq X_i$ be such that
 each $t \in X_i'$ extends some $u \in U_{i,\xi+1}$,
 every $u \in U_{i, \xi+1}$
 is extended by some $t \in X_i'$,
 and no two elements of $X_i'$ extend the same element
 of $U_{i,\xi+1}$.
Let $\zeta$ be the level of the $X_i'$.
Add $\zeta$ to the set $A$
 and set $L_{i,\zeta} := X_i'$ for each $i < \delta$.
This completes the construction and the proof.
\end{proof}

\begin{theorem}\label{thm.SDHLimpliesHL}
Let $\kappa$ be an infinite cardinal and
 let $\delta, \sigma < \kappa$ be non-zero ordinals.
Then  $\textrm{SDHL}(\delta,\sigma,\kappa)$ implies  $\textrm{HL}(\delta,\sigma,\kappa)$.
\end{theorem}
\begin{proof}
Fix $c : \bigotimes_{i < \delta} T_i \to \sigma$.
We will use Lemma~\ref{expanding_one_line}.
We claim
 that there is a level sequence
 $\vec{x} = \langle x_i \in T_i : i < \delta \rangle$
 such that for each $\zeta < \kappa$, there is a
 level sequence
 $\langle X_i \subseteq T_i : i < \delta \rangle$
 that is
 $\zeta$-$\vec{x}$-dense
 and $|c``\bigotimes_{i<\delta}X_i| = 1$.
From this, the conclusion of
 $\textrm{HL}(\delta,\sigma,\kappa)$ follows
 from Lemma~\ref{expanding_one_line}.
To show the claim, suppose towards a contradiction
 that it is false.
For each $\vec{x} \in \bigotimes_{i < \delta} T_i$
 we can associate the first $\zeta(\vec{x}) < \kappa$
 for which we cannot find a
 $\zeta(\vec{x})$-$\vec{x}$-dense
 level sequence
 $\langle X_i : i < \delta \rangle$ such that
 $|c``\bigotimes_{i<\delta} X_i| = 1$.
We can build a strictly increasing sequence
 $\langle \eta_\alpha \in \kappa : \alpha < \kappa \rangle$
 of ordinals such that for every $\alpha < \kappa$,
 $$\eta_{\alpha+1} > \zeta(\vec{x})
 \mbox{ for all } \vec{x} \in
 \bigcup_{\eta \le \eta_\alpha} \bigotimes_{i < \delta}
 T_i(\eta).$$
For $i < \delta$, set
 $$T_i^* := \bigcup_{\alpha < \kappa} T_i(\eta_\alpha).$$
Applying $\textrm{SDHL}(\delta,\sigma,\kappa)$
 to the restriction of $c$ to
 $\bigotimes_{i < \delta} T_i^*$, we find a 
 somewhere dense matrix
 $\langle X_i^* : i < \delta \rangle$ of
 the sequence of trees
 $\langle T_i^* : i < \delta \rangle$
 on which $c$ is constant.
This means that there exists
 $\alpha < \alpha' < \kappa$ and
 $\vec{x}
 = \langle x_i : i < \delta \rangle
 \in \bigotimes_{i < \delta} T_i(\eta_\alpha)$
 such that for all $i < \delta$,
 $$X_i^* \mbox{ dominates }
 \mbox{Cone}(x_i) \cap T_i(\eta_{\alpha'}).$$
It follows that
 $\langle X_i^* : i < \delta \rangle$
 is an $\eta_{\alpha'}$-$\vec{x}$-dense
 $c$-monochromatic level sequence,
 contradicting the definition of
 $\zeta(\vec{x})$ and the fact that
 $$\eta_{\alpha'} \ge \eta_{\alpha+1} > \zeta(\vec{x}).$$
This finishes the proof.
\end{proof}

 Fact \ref{fact.1treeZFC} and Theorem \ref{thm.SDHLimpliesHL} yield the following.

\begin{theorem}\label{thm.HL1treewc}
 HL$(1,k,\kappa)$ holds for each positive integer $k$
  and each infinite cardinal $\kappa$.
\end{theorem}


\section{The Halpern-L\"auchli Theorem at a measurable cardinal}\label{sec.4}

In Theorem \ref{thm.HLmbl}, we prove 
that for   any positive integer  $d$, 
assuming the existence of a cardinal $\kappa$  which is $\kappa+d$-strong cardinal (see Definition \ref{def.gammastrong}),
  it is consistent that HL$(d,\sigma,\kappa)$ holds at a measurable cardinal $\kappa$, for all $\sigma<\kappa$.
Actually,  in 
Theorem \ref{thm.generalHL} we prove an asymmetric version of the SDHL$(\delta,\sigma,\kappa)$ which holds for finitely or  infinitely many trees on a measurable cardinal.
The  requisites for that theorem are that 
certain partition relations
(see hypotheses 1 and 2 of Theorem\ref{thm.generalHL})
hold in   a model $M$  of ZF and that $\kappa$ remains measurable after forcing over $M$ with Add$(\kappa,\lambda)$, where 
 Add$(\kappa,\lambda)$ denotes the forcing which adds $\lambda$ many $\kappa$-Cohen subsets of $\kappa$ via partial functions from $\lambda\times\kappa$ into $2$ of size less than $\kappa$
and
$\lambda$ is large enough for a certain partition relation to  hold.
For infinitely many trees, the relevant partition relations hold assuming $\ad$,
 and hence in $L(\mathbb{R})$ assuming in $V$ the existence of a limit of Woodin cardinals with a measurable above.
However,
 we do not currently know of such  a model  of ZF  over which forcing to add many Cohen subsets of a measurable cardinal preserves the measurability 
(see Question \ref{q.ZFmodel} in the next section).
Nevertheless, we prove  Theorem  \ref{thm.generalHL}
in this  generality  with the optimism  that such a model will be found.

Before going any further,
 let us recall the
 standard partition relation notation:
\begin{definition}
Given cardinals $\lambda, \kappa, d, \sigma$,
 the notation
 $\lambda \rightarrow (\kappa)^d_\sigma$
 represents the statement that
 given any coloring of the size $d$ subsets of
 $\lambda$ using $\sigma$ colors,
 there is a size $\kappa$ subset of $\lambda$
 all of whose size $d$ subsets are the same color.
\end{definition}

We point out that 
 Harrington's original  forcing proof of HL$(d,k,\om)$ can be lifted, with minor modifications, to obtain HL$(d,\sigma,\kappa)$ for $\kappa$ measurable and $\sigma<\kappa$, provided, as in \cite{Shelah91}, we work in a model of ZFC in which $\kappa$ remains measurable after forcing with Add$(\kappa,\lambda)$, where $\lambda$ satisfies the partition relation $\lambda\ra(\kappa^+)^{2d}_2$.
An exposition of Harrington's original proof can be found in
 \cite{rims}.
In order for this to hold via methods of Woodin, one needs to assume the existence of a cardinal $\kappa$ which is  $\kappa+2d$-strong.
Thus, in a somewhat straightforward manner, combining known results, one may arrive at a  $\kappa+2d$-strong cardinal as an upper bound for the consistency strength of HL$(d,\sigma,\kappa)$ for $\kappa$ measurable and $\sigma<\kappa$.

However, we are interested in  the actual  consistency strength of HL$(d,\sigma,\kappa)$ for $\kappa$ measurable.
 The version of  Harrington's forcing proof for perfect subtrees of ${}^{<\om}2$  given by Todorcevic  in Theorem 8.15 of \cite{Farah/TodorcevicBK}
uses the weaker partition relation 
$\lambda\ra(\aleph_0)^d_2$, rather than the usual $\lambda\ra(\aleph_0)^{2d}_2$.
If we could lift his argument to a measurable cardinal, this would bring the consistency strength down to a $d$-strong cardinal.
One of the key lemmas in Todorcevic's  proof, though, relies strongly on the fact that $\aleph_0$ is the least infinite cardinal and could not be generalized to uncountable $\kappa$.
However,
using that $\kappa$ is measurable, we found a way around this.
We modified Harrington's argument and proved Lemma ~\ref{key_combo_lemma}
 which uses partition relations which are satisfied when $\kappa$ is measurable
 (assuming $\delta < \omega$).
This lemma aids in bringing the upper bound of 
the consistency strength of HL$(d,\sigma,\kappa)$ for $\kappa$ measurable down to a $\kappa+d$-strong cardinal in Theorem \ref{thm.HLmbl}.

\begin{definition}
Temporarily let $\mbb{P}$ be the forcing to add $\lambda$ Cohen subsets of $\kappa$.
A collection $\mc{X} \subseteq \mbb{P}$ is called \defemph{image homogenized}
 if for all $p_1, p_2 \in \mc{X}$ and $\xi$, $\alpha$, $\beta$ $\in \mbox{Ord}$,
 if $\alpha$ is the $\xi$-th element of $\dom(p_1)$ and
 $\beta$ is the $\xi$-th element of $\dom(p_2)$,
 then $p_1(\alpha) = p_2(\beta)$.
\end{definition}

Suppose that we have parameterized conditions in $\mbb{P}$
 according to $[\lambda]^\delta$.
That is, for each $\vec{\alpha} \in [\lambda]^\delta$
 we have some $p_{\vec{\alpha}}$.
For each $\vec{\alpha} \in [\lambda]^\delta$, let
 $$\langle \nu(\vec{\alpha},\xi) :
 \xi < \mbox{ot}(\dom(p_{\vec{\alpha}})) \rangle$$
 be the increasing enumeration of the elements
 of $\dom(p_{\vec{\alpha}})$.

\begin{lemma}\label{key_combo_lemma}
Let $1 \le \delta < \kappa \le \lambda$ be ordinals
 with $\kappa$ and $\lambda$ infinite cardinals and
 $\kappa$ strongly inaccessible.
Assume that for all $\mu_1, \mu_2 < \kappa$,
 $$
 \kappa \rightarrow (\mu_1)_{\mu_2}^{\delta \cdot 2}.$$
For each $\vec{\alpha} \in [\kappa]^{\delta}$,
 let $p_{\vec{\alpha}}$ be a condition in the forcing
 to add $\lambda$ many Cohen subsets of $\kappa$.
Assume that the conditions
 $p_{\vec{\alpha}}$ for $\alpha \in [\kappa]^\delta$
 are image homogenized.
Then for each $\gamma < \kappa$ there is a sequence
 $$\langle H_i \subseteq \kappa : i < \delta \rangle$$
 such that
 $(\forall i < \delta)\, \mbox{ot}(H_i) \ge \gamma$,
 $(\forall i < j < \delta)$ every element of
 $H_i$ is less than every element of $H_j$, and
 the conditions $p_{\vec{\alpha}}$ for $\vec{\alpha}
 \in \prod_{i < \delta} H_i$ are pairwise compatible.
\end{lemma}
\begin{proof}
Consider a function
 $\iota$ from $\delta \cdot 2$ to $\delta \cdot 2$ such that
 $\iota \restriction [0,\delta)$ is an increasing
 sequence of elements in $\delta \cdot 2$ and
 $\iota \restriction [\delta,\delta \cdot 2)$
 is an increasing
 sequence of elements in $\delta \cdot 2$.
Given a set $A \in [\textrm{Ord}]^{\delta \cdot 2}$,
 let $\iota[A] := \langle
 \alpha_{\iota(i)} : i < \delta \cdot 2 \rangle$,
 where $A$ enumerated in increasing order is
 $\langle \alpha_i : i < \delta \cdot 2 \rangle$.
If $\vec{\alpha}, \vec{\beta} \in [\textrm{Ord}]^\delta$
 are such that
 $(\exists A)\, \iota[A] =
 \langle \vec{\alpha}, \vec{\beta} \rangle$,
 then we say that $\vec{\alpha}$ is
 $\iota$-\textit{related} to $\vec{\beta}$.
The idea is that $\iota$ codes the
 relative ordering of the elements
 of $\vec{\alpha}$ and $\vec{\beta}$.
Given any $\vec{\alpha}, \vec{\beta}
 \in [\textrm{Ord}]^\delta$,
 there is some $\iota$ such that
 $\vec{\alpha}$ is $\iota$-related to $\vec{\beta}$.

Call a function $\iota$
 from $\delta \cdot 2$ to $\delta \cdot 2$
 \textit{acceptable} iff for all $i < \delta$,
 $$\iota(i), \iota(\delta+i)
 \in \{2 \cdot i, 2 \cdot i + 1\}.$$
The idea is that when we have our final sequence
 $\langle H_i : i < \delta \rangle$,
 if $\vec{\alpha}, \vec{\beta}
 \in \prod_{i < \delta} H_i$,
 then $\vec{\alpha}$ will be $\iota$-related to $\vec{\beta}$
 for some $\iota$ that is acceptable.
Also, as can be easily verified by the reader,
 given any acceptable $\iota$,
 there are $\vec{\alpha}, \vec{\beta}, \vec{\sigma}
 \in \prod_{i < \delta}\, H_i$
 such that
 $\vec{\alpha}$ is $\iota$-related to $\vec{\beta}$,
 $\vec{\alpha}$ is $\iota$-related to $\vec{\sigma}$, and
 $\vec{\beta}$ is $\iota$-related to $\vec{\sigma}$.

For an acceptable $\iota :
 \delta \cdot 2 \to \delta \cdot 2$,
 let $c_\iota$ be a coloring of
 subsets of $\kappa$ of size $\delta \cdot 2$
 which encodes the following information.
Given $A \in [\kappa]^{\delta \cdot 2}$,
 let $\vec{\alpha}$ and $\vec{\beta}$ be such that
 $\iota[A] = \langle \vec{\alpha}, \vec{\beta} \rangle$
 (so $\vec{\alpha}$ is $\iota$-related to $\vec{\beta}$).
The value of $c_\iota(A)$ should encode
\begin{itemize}
\item[1)] the relative ordering of the elements of
 $\dom(p_{\vec{\alpha}})$ and 
 $\dom(p_{\vec{\beta}})$, and finally
\item[2)] whether or not
 $p_{\vec{\alpha}}$ and $p_{\vec{\beta}}$
 are compatible.
\end{itemize}
This coloring $c_\iota$ can be seen to use strictly
 less than $\kappa$ colors.
In fact, since the conditions $p_{\alpha}$, $\alpha\in[\kappa]^{\delta}$ are homogenized,  their domains have the same size, say $\theta<\kappa$; so such a $c_{\iota}$ may be seen to use  at most $|{}^{\theta\cdot 2}(\theta\cdot 2)| +2$ colors.
Define $c(A)$ to be the sequence of $c_{\iota}(A)$, where $\iota$ ranges over all acceptable $\iota$, and notice that $c$ uses less than $\kappa$ many colors.
Now fix $\gamma < \kappa$.
By our assumption that
 $$(\forall \mu_1, \mu_2 < \kappa)\,
 \kappa \rightarrow (\mu_1)_{\mu_2}^{\delta \cdot 2},$$
there is  a set $H \subseteq \kappa$ of ordertype at least
 $\gamma \cdot \delta$ which is homogeneous for $c$,
and hence,
 for each  $c_\iota$.

Partition $H$ into $\delta$ pieces of size $\ge \gamma$
 that are not interleaved
 (if $P_1$ and $P_2$ are two pieces, then
 if one element of $P_1$ is less than one element of $P_2$,
 then all elements of $P_1$ are less than all elements of $P_2$).
Let the partition be
 $\langle H_i : i < \delta \rangle$.
We claim that for any two $\vec{\zeta}, \vec{\eta} \in
 \prod_{i < \delta} H_i$,
 $p_{\vec{\zeta}}$ and $p_{\vec{\eta}}$ are compatible.
Suppose, towards a contradiction,
 that there are $\vec{\zeta}$ and $\vec{\eta}$
 such that $p_{\vec{\zeta}}$ and $p_{\vec{\eta}}$ are
 incompatible.
There is an $\iota$ such that
 $\vec{\zeta}$ is $\iota$-related to $\vec{\eta}$ and
 $\iota$ is acceptable.
Since $p_{\vec{\zeta}}$ and $p_{\vec{\eta}}$ are incompatible,
 fix $\sigma_1 \not= \sigma_2$ such that
 $\nu(\vec{\zeta}, \sigma_1) = \nu(\vec{\eta}, \sigma_2)$ but
 $p_{\vec{\zeta}}(\nu(\vec{\zeta}, \sigma_1)) \not=
  p_{\vec{\eta}}(\nu(\vec{\eta}, \sigma_2))$.
The fact that it must be that $\sigma_1 \not= \sigma_2$
 follows from the fact that the $p_{\vec{\alpha}}$ are image homogenized.
Since whether or not
 $p_{\vec{\zeta}}$ and $p_{\vec{\eta}}$
 are compatible was encoded into
 $c_\iota$, any two $\vec{\alpha}, \vec{\beta} \in
 \prod_{i < \delta} H_i$ such that
 $\vec{\alpha}$ is $\iota$-related to $\vec{\beta}$
 will be such that $p_{\vec{\alpha}}$ and $p_{\vec{\beta}}$
 are incompatible.
Since $\iota$ is acceptable, fix
 $\vec{\alpha}, \vec{\beta}, \vec{\sigma} \in
 \prod_{i < \delta} H_i$ such that
 $\vec{\alpha}$ is $\iota$-related to $\vec{\beta}$,
 $\vec{\alpha}$ is $\iota$-related to $\vec{\sigma}$, and
 $\vec{\beta}$ is $\iota$-related to $\vec{\sigma}$.
By 1) of the definition of $c_\iota$
 and the fact that the conditions are image homogenized, we have
 $\nu(\vec{\alpha}, \sigma_1) = \nu(\vec{\beta}, \sigma_2)$.
Also, $\nu(\vec{\alpha}, \sigma_1) = \nu(\vec{\sigma}, \sigma_2)$
 and $\nu(\vec{\beta}, \sigma_1) = \nu(\vec{\sigma}, \sigma_2)$.
Thus, $\nu(\vec{\alpha}, \sigma_1) = \nu(\vec{\beta}, \sigma_1)$.
Now we have
 $\nu(\vec{\beta}, \sigma_1) = \nu(\vec{\beta}, \sigma_2)$,
 which is impossible because $\nu$ is an
 increasing enumeration of the elements of
 $\dom(p_{\vec{\beta}})$ and $\sigma_1 \not= \sigma_2$.
\end{proof}

Notice that 1) and 2) of the hypothesis
 of the theorem follow from
 $\lambda \rightarrow (\kappa)^{\delta \cdot 2}_\kappa$.
The conclusion of this theorem
 clearly implies $\mbox{SDHL}(\delta, \sigma, \kappa)$, and hence also $\mbox{HL}(\delta, \sigma, \kappa)$ by Theorem \ref{thm.SDHLimpliesHL}.
We point out that this proves the analogue for trees on $\kappa$ of the asymmetric Dense Set Version, Theorem 8.15 in \cite{Farah/TodorcevicBK}, for trees on $\om$.

\begin{theorem}\label{thm.generalHL}
Let $\kappa$ be an infinite cardinal, $\delta < \kappa$ a non-zero ordinal, and  $\lambda > \kappa$ a cardinal for which the following partition relations hold:
\begin{itemize}
\item[1)]
 $\lambda \rightarrow (\kappa)^{\delta}_\kappa$, and
\item[2)]
$ \kappa \rightarrow (\mu_1)_{\mu_2}^{\delta \cdot 2}$, 
 for all pairs $\mu_1, \mu_2 < \kappa$.
\end{itemize}
Assume that $\kappa$ is either $\omega$ or a measurable cardinal
 in the forcing extension to add $\lambda$ many 
 Cohen subsets of $\kappa$.
Let $T_i \subseteq {^{<\kappa}\kappa}$
 for $i < \delta$
 be regular trees, and
let $\sigma < \kappa$ be non-zero and
 $$c : \bigotimes_{i < \delta} T_i \to \sigma.$$

Then  
there are $\zeta < \kappa$
 and $\langle t_i \in T_i(\zeta) : i < \delta \rangle$,
 such that for all
 $\zeta'$ satisfying
 $\zeta \le \zeta' < \kappa$,
 there are $\zeta''$ satisfying
 $\zeta' \le \zeta'' < \kappa$
 and $\langle X_i \subseteq T_i(\zeta'') : i < \delta \rangle$
 such that, for each  $i < \delta$, $X_i$
 dominates
 $T_i(\zeta') \cap
 \textnormal{Cone}(t_i)$, and 
 $$|c`` \bigotimes_{i < \delta} X_i | = 1.$$
Furthermore, we can ensure that if
 $c``\bigotimes_{i < \delta} X_i = \{0\}$, then
 $(\forall i < \delta)\, t_i = \emptyset$.

In particular,  $\mbox{HL}(\delta, \sigma, \kappa)$  holds.
\end{theorem}

\begin{proof}
Note that $\kappa$ is strongly inaccessible.
Let $\mbb{P}$ be the poset to add
 $\delta \times \lambda$ many Cohen subsets of $\kappa$
 presented in the following way:
 $p \in \mbb{P}$ iff $p$ is a size $<\kappa$
 partial function from $\lambda \times \delta$
 to ${^{<\kappa}\kappa}$ and
 for all $(\alpha,i) \in \dom(p)$,
 $p(\alpha,i) \in T_i$.
For the ordering of $\mbb{P}$,
 $q \le p$ iff
 $$\dom(q) \sqsupseteq \dom(p)
 \mbox{ and } (\forall x \in \dom(p))\,
 q(x) \sqsupseteq p(x).$$
If each $T_i = {^{<\kappa} \kappa}$,
 then it is trivial to see that this forcing
 is equivalent to adding $\lambda$
 Cohen subsets of $\kappa$.
If not, then as long as each $T_i$
 does not have any isolated branches
 or leaf nodes (which we assume), then because
 no path leaves tree $T_i$ at a limit level
 $< \kappa$,
 one shows that the forcings are still equivalent.
Without loss of generality,
 $(\forall p \in \mbb{P})$
 the size of $\dom(p) \cap (\lambda \times \{i\})$
 does not depend on $i$.
Similarly, we may assume
 $(\forall p \in \mbb{P})\,$
 the length of the sequence $p(x)$
 does not depend on $x \in \dom(p)$.

Let $\dot{G}$ be the canonical name for the
 generic object (as opposed to the generic filter).
In particular,
 $1 \forces \dot{G} :
 \check{\lambda} \times \check{\delta}
 \to {^{\check{\kappa}} \check{\kappa} }$.
Since
 $1 \forces (\check{\kappa}$ is a measurable cardinal$)$,
 let $\dot{\mc{U}}$ be such that
 $1 \forces (\dot{\mc{U}}$ is a $\check{\kappa}$-complete
 ultrafilter on $\check{\kappa}$).

For each $\vec{\alpha} \in [\lambda]^\delta$,
 let $p_{\vec{\alpha}} \in \mbb{P}$ be a condition
 and $\sigma_{\vec{\alpha}} < \sigma$ be a color such that
 if $\langle \alpha_i : i < \delta \rangle$
 is the increasing enumeration of $\vec{\alpha}$, then 
 $$p_{\vec{\alpha}} \forces
 \{ \zeta < \check{\kappa} :
 c( \langle
 \dot{G}(\alpha_i,i) \restriction \zeta :
 i < \delta \rangle ) =
 \check {\sigma}_{\vec{\alpha}} \} \in \dot{\mc{U}}.$$
We may assume that if $\sigma_{\vec{\alpha}} = 0$, then
 $p_{\vec{\alpha}} = 1$.
This is because if $1$ does not force the color
 of $\dot{\mc{U}}$ many levels to be $0$, then
by the nature of the forcing relation
 and since $\dot{\mc{U}}$ is an ultrafilter,
 there is some condition $p$ such that
$$p \forces
 \{ \zeta < \check{\kappa} :
 c( \langle
 \dot{G}(\alpha_i,i) \restriction \zeta :
 i < \delta \rangle ) \not= 0 \} \in \dot{\mc{U}}.$$
This will establish the last sentence of the statement
 of the theorem.
Then, since $\sigma < \kappa$ and
 $\dot{\mc{U}}$ is $\kappa$-complete (in the extension),
 there is $p' \le p$ and $\sigma \not= 0$ such that
 $$p' \forces
 \{ \zeta < \check{\kappa} :
 c( \langle
 \dot{G}(\alpha_i,i) \restriction \zeta :
 i < \delta \rangle ) = \sigma \} \in \dot{\mc{U}}.$$
Then set $p_{\vec{\alpha}} = p'$ and
 $\sigma_{\vec{\alpha}} = \sigma$ and we are done.
We may also assume, by possibly making the conditions
 $p_{\vec{\alpha}}$ stronger,
 that for each
 $\vec{\alpha} \in [\lambda]^\delta$
 and $i < \delta$, that
 $$(\alpha_i,i) \in \dom(p_{\vec{\alpha}}).$$
There is a coloring
 $\tilde{c} : [\lambda]^\delta \to \kappa$
 such that after we apply the partition relation
 $\lambda \rightarrow (\kappa)^\delta_\kappa$,
 we get a set $H \in [\lambda]^\kappa$
 such that the following are satisfied:
\begin{itemize}
\item[1)] the sequence
 $\langle(\xi, p_{\vec{\alpha}}(\nu(\vec{\alpha},\xi))) :
 \xi < \mbox{ot}(\dom(p_{\vec{\alpha}})) \rangle$
 does not depend on $\vec{\alpha} \in [H]^\delta$
 and therefore the set of $p_{\vec{\alpha}}$ for $\vec{\alpha} \in [H]^\delta$
 are image homogenized,
\item[2)] the value $\sigma_{\vec{\alpha}}$ does not depend on $\vec{\alpha} \in [H]^\delta$, and
\item[3)] for a fixed $i < \delta$, the sequence
 $\langle p_{\vec{\alpha}}(\alpha_i,i) \in T_i
 : i < \delta \rangle$
 does not depend on
 $\vec{\alpha} \in [H]^\delta$.
\end{itemize}
Such a coloring $\tilde{c}$ may be defined as follows.
For $\vec{\al}\in[\lambda]^{\delta}$,
define $\tilde{c}(\vec\al)$ to be the triple  sequence
consisting of the following three sequences:
the sequence from 1);
 the value $\sigma_{\vec\al}$   that $p_{\vec\al}$ forces $c(\lgl \dot{G}(\al_i,i)\re\zeta:i<\delta\rgl)$ to equal for $\dot{\mathcal{U}}$ many levels $\zeta$;
and the sequence
$\lgl \xi_i:i<\delta\rgl$,
where $\xi_i$ is the ordinal such  that 
$\al_i$ appears as the $\xi_i$-th member 
 of 
$\mbox{ot}(\dom(p_{\vec{\alpha}}))$.

For each $i < \delta$,
 let $t_i \in {^{<\kappa}\kappa}$ be the unique value
 of the $p_{\vec{\alpha}}(i,\alpha_i)$'s
 for $\vec{\alpha} =
 \langle \alpha_i : i < \delta \rangle
 \in [H]^\delta$.
Let $\zeta$ be the common length of each
 $t_i$.
Let $\sigma' < \sigma$ be the unique value of
 $\sigma_{\vec{\alpha}}$ for
 $\vec{\alpha} \in [H]^\delta$.
Note that if $\sigma' = 0$,
 then each $p_{\vec{\alpha}}$ for
 $\vec{\alpha} \in [H]^\delta$ equals $1$,
 so $(\forall i < \delta)\, t_i = \emptyset.$

Now pick an arbitrary $\zeta'$ satisfying
 $\zeta \le \zeta' < \kappa$.
Let $\gamma < \kappa$ be a cardinal such that for each
 $i < \delta$, the cardinality of the set
 $T(\zeta') \cap \mbox{Cone}(t_i)$
 is $\le \gamma$.
Now apply Lemma~\ref{key_combo_lemma}
 to get a sequence
 $\langle H_i \subseteq H : i < \delta \rangle$
 such that
 $(\forall i < \delta)\, \mbox{ot}(H_i) \ge \gamma$,
 $(\forall i < j < \delta)$ every element of
 $H_i$ is less than every element of $H_j$, and
 the conditions $p_{\vec{\alpha}}$ for $\vec{\alpha}
 \in \prod_{i < \delta} H_i$ are pairwise compatible.
To apply this lemma, we needed hypothesis 2) of this theorem
 (to hold in the ground model).
We could instead apply Lemma~\ref{key_combo_lemma}
 in the forcing extension as long as 2) holds in the extension,
 and the sequence $\langle H_i : i < \delta \rangle$
 would be in the ground model because the forcing is
 ${{<}\kappa}$-closed.

For each $i < \delta$,
 let $S_i := T_i(\zeta') \cap \textrm{Cone}(t_i)$,
 the set of successors of $t_i$ in $T_i$.
For each $i < \delta$ and $t \in S_i$,
 pick $\alpha_{i,t} \in H_i$ so that
 every element of $S_i$
 is mapped to a unique element of $H_i$.
This is possible because the cardinality of
 $S_i$ is $\le \gamma$
 and $\gamma$ is $\le$ the ordertype of $H_i$.
For $i < \delta$,
 let $A_i :=
 \{ \alpha_{i,t} :
 t \in S_i \}
 \subseteq H_i$.
Let $\mc{X} \subseteq \prod_{i < \delta} H_i$
 be the set
 $\mc{X} := \prod_{i < \delta} A_i$.
The conditions $p_{\vec{\alpha}}$ for
 $\vec{\alpha} \in \mc{X}$ are pairwise compatible.

Let $p \in \mbb{P}$ be a condition which extends
 $\bigcup \{ p_{\vec{\alpha}} :
 \vec{\alpha} \in \mc{X} \}$
 and
 for all $i < \delta$
 and $t \in S_i$,
 $p(\alpha_{i,t},i) = t.$
To see that there exists such a $p$,
 first note that
 $|\mc{X}| < \kappa$ and
 the conditons $p_{\vec{\alpha}} \in \mc{X}$
 are pairwise compatible,
 therefore by the nature of $\mbb{P}$,
 $\bigcup_{\vec{\alpha} \in \mc{X}} p_{\vec{\alpha}}$
 is an element of $\mbb{P}$.
Second,
 since $(\forall \vec{\alpha} \in \mc{X})
 (\forall i < \delta)\,
 p_{\vec{\alpha}}(\alpha_i,i)
 = t_i$,
 we may define $p$ so that
 $p(i,\alpha_{i,t}) = t$
 for all $i < \delta$ and $t \in S_i$
 and this will not clash with the $p_{\vec{\alpha}}$.

Since $|\mc{X}| < \kappa$,
 $1 \forces (\dot{\mc{U}}$ is a $\check{\kappa}$-complete ultrafilter$)$,
 and by the hypothesis on each pair
 $(p_{\vec{\alpha}}, \sigma_{\vec{\alpha}})$,
 we have that
 $p \forces$ there are arbitrarily large
 levels $\zeta'' < \check{\kappa}$ such that for each
 $\vec{\alpha} \in \check{\mc{X}}$, we have
 $$\check{c}( \langle
  \dot{G}(\alpha_i,i) \restriction \zeta''
  : i < \check{\delta} \rangle )
 = \check{\sigma}'.$$
We may extend $p$ to a condition
 $p' \le p$ as well as get a level
 $\zeta'' \in [\zeta',\kappa)$
 such that the following are satisfied:
\begin{itemize}
\item[1)]
 $p' \forces$
 for each
 $\vec{\alpha} =
 \langle \alpha_i : i < \delta \rangle
 \in \check{\mc{X}}$, we have
$$\check{c}( \langle
  \dot{G}(\alpha_i,i) \restriction \zeta'' :
  i < \check{\delta} \rangle )
 = \check{\sigma}'.$$
\item[2)] for each
 $i < \delta$ and each
 $t \in S_i$,
 there is a unique $\tilde{t}
 \in T_i(\zeta'')$ such that
 $p' \forces \dot{G}(\check{\alpha}_{i,t}, \check{i})
 \sqsupseteq \check{\tilde{t}}$.
\end{itemize}
For each $i < \delta$,
 let $$X_i := \{ \tilde{t} : t \in S_i \}.$$
We have that each $X_i$
 dominates $S_i$ and
 $$c `` \bigotimes_{i<\delta} X_i = \{ \sigma' \}$$
 (because the coloring $c$ is in the ground model,
 so it is absolute).
\end{proof}

The following theorem of \Erdos\ and Rado provides  many examples  of cardinals $\lambda$ which 
satisfy the hypothesis in Theorem \ref{thm.generalHL}, 
 when  $\delta$  is finite.

\begin{theorem}[\Erdos-Rado, \cite{Erdos/Rado56}]\label{thm.ER}
For $r\ge 0$ finite and $\kappa$ an infinite cardinal,
$\beth_r(\kappa)^+\ra(\kappa^+)^{r+1}_{\kappa}$.
\end{theorem}

In particular, if GCH holds, 
then for finite $d\ge 2$,
$\kappa^{+d}\ra (\kappa^+)^d_{\kappa}$.

\begin{definition}[\cite{KanamoriBK}]\label{def.gammastrong}
Given a cardinal $\kappa$ and an ordinal $d$,
$\kappa$ is {\em $\kappa+d$-strong} if there is an elementary embedding $j:V\ra M$ with critical point $\kappa$ such that $V_{\kappa+d}= M_{\kappa+d}$.
\end{definition}

\begin{theorem}\label{thm.HLmbl}
Let $d\ge 1$ be any  finite integer and  suppose $\kappa$ is a $\kappa+d$-strong cardinal in a model  $V$ of ZFC satisfying GCH.
Then there is a forcing extension in which $\kappa$ remains measurable and 
 HL$(d,\sigma,\kappa)$ holds, for all $\sigma<\kappa$.
\end{theorem}

\begin{proof}
Let $\lgl\kappa_{\al}:\al<\kappa\rgl$  enumerate all the  strongly inaccessible cardinals  in $V$ below $\kappa$.
Let  $\bP_{\kappa}$ denote  the  $\kappa$-length reverse Easton support iteration of Add$(\kappa_{\al},\kappa_{\al}^{+d})$, and let $G_{\kappa}$ be $\bP_{\kappa}$-generic over $V$.
Then in $V[G_{\kappa}]$,
$\kappa$ is still measurable by a standard lifting of the embedding,  and
GCH holds at and above $\kappa$
since $|\bP_{\kappa}|=\kappa$.
Thus, in $V[G_{\kappa}]$, the partition relation $\kappa^{+d}\ra (\kappa)^d_{\kappa}$ holds.
Since $\kappa$ is measurable in $V[G_{\kappa}]$, the partition relations $\kappa\ra(\mu_1)^{d\cdot 2}_{\mu_2}$ holds for all pairs $\mu_1,\mu_2<\kappa$.
Let $\bQ$ denote Add$(\kappa,\kappa^{+d})$ in $V[G_\kappa]$, and 
let $H$ be $\bQ$-generic over $V[G_\kappa]$.
By an unpublished  result of Woodin (see for instance \cite{Friedman/Thompson08} for a proof)
$\kappa$ remains measurable in $V[G_{\kappa}][H]$, since $\kappa$ is $\kappa+d$-strong in $V$.
Thus, the hypotheses of Theorem \ref{thm.generalHL} 
 are satisfied in $V[G_{\kappa}]$,
and therefore HL$(d,\sigma,\kappa)$ holds in $V[G_\kappa]$.
\end{proof}

We conclude this section by pointing out that, in place of Add$(\kappa_{\al},\kappa_{\al}^{+d})$, one could use  a reverse Easton iteration of 
$\kappa_{\al}^{+d}$-products with ${{<}\kappa_\alpha}$-support of  $\kappa_{\al}$-Sacks forcing, as in  \cite{Friedman/Thompson08}, to achieve a model with $\kappa$ measurable and HL$(d,\sigma,\kappa)$ holding.
However, the homogeneity argument in
 the body of Theorem~\ref{thm.generalHL}
 would use $\kappa^+$ colors, and so we would 
need to start with a cardinal $\kappa$ which is $(\kappa+d+1)$-strong in the ground model
 to get an analogue of Theorem~\ref{thm.HLmbl}.
It should be noted though that for trees of height $\om$, Harrington's original forcing proof can also be modified to use Sacks forcing; the larger $\lambda$ needed to accommodate the homogeneity argument
 would need to satisfy $\lambda \rightarrow (\omega)^{2d}_{2^\omega}$.


\section{Closing Comments and Open Problems}\label{sec.6}

The following conjectures and questions are motivated by the results in the previous sections and their comparisons with results in \cite{Shelah91}, \cite{Hajnal/Komjath97} and  \cite{Dzamonja/Larson/MitchellQ09}.
We have found upper bounds for the consistency strengths of HL$(d,\sigma,\kappa)$ for $\kappa$ measurable and are interested in  the exact consistency strength.

\begin{conjecture}
For finite $d\ge 2$,
 $\kappa$ measurable and $\sigma<\kappa$,
 the consistency strength of  HL$(d,\sigma,\kappa)$   is
a $\kappa+d$-strong  cardinal.
\end{conjecture}

If  the conjecture turns out to be true, then there must be a positive answer to the next question.

\begin{question}\label{op1}
Given $d \ge 1$,
is there  a model of ZFC
in which 
there is measurable cardinal $\kappa$ such that 
HL$(d,\sigma,\kappa)$ holds for all $\sigma<\kappa$, but HL$(d+1,\sigma,\kappa)$ fails for some $2 \le \sigma<\kappa$?
\end{question}

In Section 8 of \cite{Dzamonja/Larson/MitchellQ09}, it is mentioned that  a model satisfying the hypotheses of  Theorem 2.5 in \cite{Dzamonja/Larson/MitchellQ09} can be constructed,
assuming the existence of a measurable cardinal
$\kappa$
 such that $o(\kappa)=\kappa^{+2m+2}$.
We conjecture that the form of Halpern-\Lauchli\  in \cite{Shelah91} and
\cite{Dzamonja/Larson/MitchellQ09} is strictly stronger than the form HL$(d,\sigma,\kappa)$.

\begin{conjecture}
Let $d\ge 2$ be a finite number.
The consistency strength of 
HL$(d,\sigma,\kappa)$ for $\kappa$ measurable
 is strictly  less than the consistency strength of
 Theorem 2.5 in \cite{Dzamonja/Larson/MitchellQ09}
 for coloring $d$-sized antichains.
\end{conjecture}

Further,  Dzamonja, Larson, and Mitchell point out that
Theorem 2.5 in \cite{Dzamonja/Larson/MitchellQ09}
is not a consequence of   any large cardinal assumption.
This follows from results of Hajnal and Komj\'{a}th 
in \cite{Hajnal/Komjath97}, a consequence of which is that there is a forcing of size $\aleph_1$  after which there is a coloring of the pairsets on ${}^{<\kappa}2$ for which there is no strong subtree homogeneous for the coloring.
However, that theorem  of Hajnal and Komj\'{a}th  does not seem to immediately provide a counterexample to HL$(d,2,\kappa)$, and so we ask the following.

\begin{question}
For $d\ge 2$, is there a large cardinal assumption on $\kappa$ which implies HL$(d,2,\kappa)$ holds?
\end{question}

In fact, the following is open:
\begin{question}
Is HL$(2,2,\kappa)$ for some $\kappa > \omega$ a theorem of $\zfc$?
\end{question}

In Section \ref{sec.3} we showed that all the variants considered in this paper are
 equivalent as long as $\kappa$ is weakly compact,
and showed that HL$(1,k,\kappa)$ holds when $\kappa$ is infinite and $k$ is finite.
In Section 8 of \cite{Dzamonja/Larson/MitchellQ09},
an argument is provided showing that any cardinal 
 $\kappa$ satisfying their   Theorem 2.5  for $m\ge 2$ must be weakly compact.
However, that argument does not  seem to apply to our situation
and so we ask the following.

\begin{question}
For $d \ge 2$, is the statement ``$\kappa\ge \aleph_1$ and  HL$(d,2,\kappa)$ holds" equiconsistent with some large cardinal?
\end{question}

Lastly, we proved Theorem \ref{thm.generalHL} for the general case of infinitely many trees with the optimism that a model of ZF satisfying the conditions can be found, and hence  ask the following.

\begin{question}
\label{q.ZFmodel}
Is there a model  of ZF
satisfying the partition relations stated in the hypotheses of Theorem \ref{thm.generalHL} 
for some infinite $\delta<\kappa$
such that after forcing with Add$(\kappa,\lambda)$,
 $\kappa$ is measurable?
\end{question}

\bibliographystyle{amsplain}
\bibliography{references}

\end{document}